\documentclass[pdflatex,sn-mathphys-num]{sn-jnl1}

\usepackage{graphicx}%
\usepackage{multirow}%
\usepackage{amsmath,amssymb,amsfonts}%
\usepackage{amsthm}%
\usepackage{mathrsfs}%
\usepackage[title]{appendix}%
\usepackage{xcolor}%
\usepackage{textcomp}%
\usepackage{manyfoot}%
\usepackage{booktabs}%
\usepackage{algorithm}%
\usepackage{algorithmicx}%
\usepackage{algpseudocode}%
\usepackage{listings}%

\theoremstyle{thmstyleone}%
\newtheorem{theorem}{Theorem}[section]

\newtheorem{proposition}[theorem]{Proposition}

\newtheorem{definition}[theorem]{Definition}

\newtheorem{corollary}[theorem]{Corollary}
\numberwithin{equation}{section}
\advance\hoffset-1truecm\relax
\def\n{\noindent}

\def\ve{\varepsilon}
\def\va{\varphi}

\def\O{\Omega}

\def\cn{\mathbb C^n}
\def\om{\omega}

\def\wed{\wedge}

\begin{document}

\title[Continuity of the complex Monge-Amp\`ere operator on compact Hermitian manifolds]{Continuity of the complex Monge-Amp\`ere operator on compact Hermitian manifolds}


\author[1]{\fnm{Le Mau Hai} }\email{mauhai@hnue.edu.vn}
\equalcont{These authors contributed equally to this work.}
\author[2]{\fnm{Nguyen Van Phu} }\email{phunv@epu.edu.vn}
\equalcont{These authors contributed equally to this work.}

\author*[3]{\fnm{Trinh Tung} }\email{tungtrinhvn@gmail.com}
\equalcont{These authors contributed equally to this work.}

\affil[1]{\orgdiv{Department of Mathematics}, \orgname{Hanoi National University of Education}, \orgaddress{\street{136 Xuan Thuy}, \city{Hanoi},  \country{Vietnam}}}

\affil[2]{\orgdiv{Department of Mathematics, Faculty of Natural Sciences}, \orgname{Electric Power University}, \orgaddress{\street{235 Hoang Quoc Viet}, \city{Hanoi},  \country{Vietnam}}}

\affil*[3]{\orgdiv{Department of Mathematics, Faculty of Natural Sciences}, \orgname{Electric Power University}, \orgaddress{\street{235 Hoang Quoc Viet}, \city{Hanoi},  \country{Vietnam}}}

\abstract{ {\fontsize{11 pt}{1}\selectfont	In this note, we establish several results concerning the continuity (or weak convergence) of the complex Monge-Amp\`ere operator on compact Hermitian manifolds. At the end of this note, we find a weak solution of the complex Monge-Amp\`ere equation on a compact Hermitian manifold under the assumption of the existence of a smooth subsolution.}   }


\maketitle

	\section{Introduction}
	\subsection{Background}
{\fontsize{11 pt}{1}\selectfont	
 Let $(X,\om)$ be a compact Hermitian manifold of complex dimension $n$, equipped with the Hermitian fundamental form $\om$ given in local coordinates by  
$ \om=\frac{i}{2}\sum\limits_{\alpha,\beta} a_{\alpha,\bar{\beta}} dz_{\alpha}\wed d\bar{z}_{\beta},$
 where $(a_{\alpha,\bar{\beta}})$ is a positive definite Hermitian matrix. In the case $d\om=0$, then $\om$ is called a K\"ahler form and $(X,\om)$ is said to be a compact K\"ahler manifold. The smooth volume form $dV$ associated to this Hermitian metric is given by the $n$-th wedge product $dV=\om^n$. The study of the continuity of the complex Monge-Amp\`ere operator $(dd^c\bullet)^n$ plays an important role in many problems of pluripotential theory. The original results in the direction of this research from the seminal work of Betford and Taylor in \cite{BeTa82}. In this paper, Bedford and Taylor proved that if a sequence of locally bounded plurisubharmonic functions $(u_j)\subset PSH(\O)\cap L^{\infty}_{loc}(\O)$, where $PSH(\O)$ denotes the set of plurisubharmonic functions on an open subset $\O\subset\cn$, is decreasing to a locally bounded plurisubharmonic function $u$ then $(dd^c u_j)^n$ is weakly convergent to $(dd^c u)^n$. They also established a result about weak convergence of the sequence of the complex Monge-Amp\`ere $(dd^c u_j)^n$ to $(dd^c u)^n$ under the assumption that $(u_j)$ is an increasing sequence and a.e. convergent concerning the Lebesgue measure $dV_{2n}$ to a locally bounded plurisubharmonic $u$. In summary, by relying on the results obtained by Bedford-Taylor, we observe that the Monge-Amp\`ere operator $(dd^c \bullet)^n$ is continuous in weak*-topology along monotonic sequences in the class of locally bounded plurisubharmonic functions.

 It is worth noting, however, that monotone convergence of a sequence of plurisubharmonic functions \((u_j)\) to a psh function \((u)\) does not always occur. This raises the natural question of whether the monotonicity assumption can be replaced by a weaker condition, such as convergence in $L^p$, which still guarantees the convergence in the weak*-topology of the sequence of complex Monge-Amp\`ere operators $(dd^c u_j)^n$ to $(dd^c u)^n$. Unfortunately, that does not happen. Due to a counterexample given in \cite{Ce83}, Cegrell showed that there exists a sequence of continuous plurisubharmonic functions $(u_j)$ in $\mathbb{C}^2$ converging in the $L^p, 1\leq p<+\infty$-topology to a plurisubharmonic function $u$, but the corresponding sequence of Monge-Amp\`ere operators $(dd^c u_j)^2$ does not converge in the weak*-topology to $(dd^c u)^2$. Thus, convergence in the $L^p$-topology for $p \geq 1$ does not imply the weak continuity of the sequence of Monge-Amp\`ere operators.

 At this point, one seeks a condition weaker than the monotone convergence of sequences of plurisubharmonic functions, yet strong enough to establish the continuity of the Monge-Amp\`ere operator sequence in the weak*-topology.
In this direction, Xing, in his 1996 paper \cite{Xi96}, introduced the notion of convergence in \(C_n\)-capacity for sequences of plurisubharmonic functions. He proved that if \((u_j)\) is a sequence of plurisubharmonic functions, uniformly bounded on an open subset \(\Omega \subset \mathbb{C}^n\), and converging to a locally bounded plurisubharmonic function \(u\) in \(C_n\)-capacity, then the corresponding Monge-Amp\`ere measures \((dd^c u_j)^n\) converge to \((dd^c u)^n\) in the weak\(^*\)-topology. Also in that paper, based on the concept of convergence in $C_n$-capacity, Xing obtained several other forms of weak convergence for sequences of Monge-Amp\`ere measures. For more details on these results, we refer readers to \cite{Xi96}. Next, in \cite{Xi08}, relying on the concept of convergence in $C_n$-capacity, Xing has provided deeper results concerning the relationship between the convergence in $C_n$-capacity of a sequence of uniformly bounded plurisubharmonic functions  $(u_j)$ to a locally bounded plurisubharmonic function $u$, and the weak convergence of the sequence of corresponding Monge-Amp\`ere measures  $(dd^c u_j)^n$ and  $(dd^c u)^n$. In fact, Xing found a type of convergence of the complex Monge-Amp\`ere measures which is essentially equivalent to convergence in $C_n$-capacity of plurisubharmonic functions, respectively (see Theorems 2.8, 2.11, 2.14  and Corollaries 2.10, 2.12 in \cite{Xi08}).

Note that, from the monotone convergence of a locally bounded sequence of plurisubharmonic functions, it is not difficult to deduce the convergence of this sequence in $C_n$-capacity. This shows that convergence in $C_n$-capacity is a weaker condition than monotone convergence, yet it is still effective for establishing the weak convergence of Monge-Amp\`ere measure sequences. It should be noted that it was precisely the discovery of the weak convergence of the Monge-Amp\`ere measures of bounded plurisubharmonic functions on hyperconvex domains $\O\subset\cn$, made between 1998 and 2004, that led Cegrell in \cite{Ce98} and \cite{Ce04} to show that the Monge-Amp\`ere operator $(dd^c \bullet)^n$ is well-defined on certain classes of unbounded plurisubharmonic functions. He introduced and investigated some classes of plurisubharmonic functions such as $\mathcal{E}_p(\Omega), \mathcal{F}_p(\Omega), \mathcal{E}(\Omega), \mathcal{F}(\Omega)$, and, at the same time, he investigated the solvability of the Monge-Amp\`ere equations within these classes. In \cite{Ce04}, Cegrell showed that if $\Omega \subset \mathbb{C}^n$ is a hyperconvex domain, then $\mathcal{E}(\Omega)$ is the biggest class on which the Monge-Amp\`ere operator is well-defined and continuous under decreasing sequences of plurisubharmonic functions in that class. For more details on the aforementioned issues, we refer readers to the concepts and results presented in \cite{Ce98} and \cite{Ce04}, respectively.

A natural question that arises here is whether the convergence in $C_n$-capacity of plurisubharmonic functions in the above-mentioned Cegrell's classes implies the weak convergence of the corresponding Monge-Amp\`ere measures $(dd^c \bullet)^n$. The study of the weak convergence of the complex Monge-Amp\`ere operator in these classes was initiated by Cegrell in a 2001 manuscript (see \cite{Ce01}) and was completed and published in \cite{Ce12} in 2012. In \cite{Ce12}, he proved that if $(u_j)$ is a sequence of plurisubharmonic functions in $\mathcal{E}(\O)$, where $\O$ is a hyperconvex domain in $\cn$, $u\in\mathcal{E}(\O)$ and there exists $u_0\in\mathcal{E}(\O)$ such that $u_j\geq u_0$ for all $j\geq 1$, $(u_j)$ is convergent to $u$ in $C_n$-capacity then $(dd^c u_j)^n$ is weakly convergent to $(dd^c u)^n$. Since then, the problem of studying the weak convergence of the Monge-Amp\`ere operator sequence in Cegrell classes under the influence of convergence in $C_n$-capacity has attracted the attention of several authors.

In \cite{Xi08}, Xing investigated the weak convergence of a sequence of Monge-Amp\`ere operator $(dd^c\bullet)^n$ in the class $\mathcal{F}(\O)$. Pham Hoang Hiep in \cite{Hi08}, by removing the condition $ u_j \geq u_0 $, it was shown in \cite{Ce12} that if $ (u_j) \subset \mathcal{E}(\Omega)$ and $ u_j \rightarrow u$ in $C_n$-capacity, then
$$\lim\inf\limits_{j\to\infty} (dd^c u_j)^n\geq 1_{\{u>-\infty\}}(dd^c u)^n.$$
\n The results discussed above concerning convergence in \(C_n\)-capacity and the associated weak convergence of Monge--Amp\`ere measures have thus far been formulated in the local setting, i.e., on open subsets of \(\mathbb{C}^n\).

 In this work, we aim to extend this line of research to the global setting, with a particular focus on the connection between convergence in \(C_n\)-capacity and the weak convergence of Monge--Amp\`ere operators on complex manifolds,
specifically, on compact K\"ahler manifolds and, more generally, on compact Hermitian manifolds. These developments are relatively recent and serve as the principal motivation for the present paper. We begin by investigating the weak convergence of Monge--Amp\`ere measures under convergence in capacity  on a compact K\"ahler manifold \((X, \omega)\) of complex dimension \(n\). Since every plurisubharmonic function on $X$ is necessarily constant, the appropriate framework in this context is to work with \emph{quasi-plurisubharmonic} functions, also known as \(\omega\)-plurisubharmonic functions (abbreviated as \(\omega\)-psh functions). This class of functions was introduced and studied in detail, and some fundamental results for this class were obtained in \cite{GZ05}. By $PSH(\om)$ we denote the class of $\om$-psh on $(X,\om)$. We recall some main results in \cite{GZ05}. If $u\in PSH(\om)\cap L^{\infty}(X)$ then in\cite{GZ05} they showed that there exists a Monge-Amp\`ere measure $\om_u^n= (\om +dd^c u)^n$ and it is continuous on decreasing sequences ( also see \cite{KN15} for arguments on the existence of Monge-Amp\`ere measure $\om_u^n$ in the situation $(X,\om)$ is a compact Hermitian manifold). Definition 3.4 in \cite{GZ05} deals with the notion of capacity of a Borel subset $E\subset X$ on a compact K\"ahler manifold $(X,\om)$. Corollary 3.8 gives the quasi-continuity of $\om$-psh. Remark that here we only recall some notions concerning the class of $\om$-psh functions. The precise definition of $\omega$-plurisubharmonic functions as well as the notion of capacity in the case where $(X, \omega)$ is a compact Hermitian manifold will be presented in detail in the following section of the paper.

Next, in \cite{GZ07}, Guedj and Zeriahi introduced and studied the class $\mathcal{E}(X,\om)$. They showed that the complex Monge-Amp\`ere operator $(\om +dd^c\bullet)^n$ is well defined on the class $\mathcal{E}(X,\om)$ of $\om$-psh with finite weighted Monge-Amp\`ere energy. The class \(\mathcal{E}(X, \omega)\) is the largest class of \(\omega\)-plurisubharmonic (\(\omega\)-psh) functions on which the complex Monge--Amp\`ere operator is well defined and for which the comparison principle holds. For detailed results concerning this class, we refer the reader to \cite{GZ07}.
The study of convergence with respect to the global capacity \(Cap_{\omega}\) on compact K\"ahler manifolds was initiated in \cite{Hi08a,Hi10}. In that paper, Hiep established several equivalent characterizations for convergence in \(Cap_{\omega}\) of a uniformly bounded sequence \((u_j)\) of \(\omega\)-psh functions to a given \(\omega\)-psh function \(u\) (see Theorem 2.1 and Lemma 2.3 in \cite{Hi08a}).

Subsequently, the study of the weak continuity of the Monge-Amp\`ere operator $(\omega + dd^c\bullet)^n$ concerning convergence in $Cap_{\omega}$ on the class $\mathcal{E}(X,\omega)$ was carried out by Xing in \cite{Xi08}. Up until 2008, the results obtained by Xing in \cite{Xi08,Xi09} represented the most profound and significant contributions to this line of research. Subsequently, in 2012, Dinew and Hiep further advanced the study of convergence in capacity on compact K\"ahler manifolds within the class $\mathcal{E}(X, \omega)$ (see \cite{DH12}). Their work synthesizes and extends several results previously established in \cite{Xi08}.  The results mentioned above are results concerning the continuity of the Monge-Amp\`ere operator for $\omega$-plurisubharmonic functions related to convergence in the capacity $Cap_{\omega}$ on compact K\"ahler manifolds.

Note that, when studying the continuity of the Monge--Amp\`ere operator $(dd^c u_j)^{n}$ on Cegrell classes $\mathcal{E}(\Omega)$, it is usually necessary to assume that the sequence of functions $(u_j)$ is bounded from below by a function $u_0$. However, when considering the continuity problem for the Monge--Amp\`ere operator $(\omega + dd^c \bullet)^{n}$ on compact K\"ahler manifolds, as in the papers of Xing (refer \cite{Xi08,Xi09}) or  Dinew and Hiep in \cite{DH12}, the authors were able to dispense with this condition, since finding such a lower bounding function is not always feasible.  Now, regarding compact Hermitian manifolds, what results have been obtained in this direction of research? Unfortunately, results in this direction for the case $(X,\omega)$, which is a compact Hermitian manifold, remain quite limited up to now. Only recently one knew the study of weak convergence of the Monge-Amp\`ere operator for the class of $\omega$-plurisubharmonic functions appeared, notably in the paper by Ko{\l}odziej and Ngoc Cuong Nguyen in \cite{KN25}. Regrettably, in this paper, under the assumptions of convergence in capacity, the authors were only able to obtain weak convergence of a subsequence $(\omega^n_{u_{j_k}})$ of the original sequence $(\omega^n_{u_j})$ (see Theorem 1.1 in \cite{KN25}). Inspired by the results in \cite{KN25}, in this paper, we aim to provide additional results on the weak convergence of the Monge-Amp\`ere operator on compact Hermitian manifolds for the class of $\omega$-plurisubharmonic functions, in relation to convergence in $Cap_{\omega}$.

\subsection{Organization}\noindent The paper is presented in four sections and organized as follows.\\ In Section \ref{sec2}, we present important results concerning the class of $\omega$-plurisubharmonic functions ($\omega$-psh) that will be used in the subsequent parts of the paper. This includes the definition of the $\omega$-psh class, the Monge-Amp\`ere operator for bounded $\omega$-psh functions, the notion of weak convergence of a sequence of Monge-Amp\`ere operators, the definition of the capacity $Cap_{\omega}$ and the convergence of a sequence $(u_j)$ of $\omega$-psh functions to an $\omega$-psh function in $Cap_{\om}$, as well as several other related concepts. Section \ref{sec3} is devoted to investigating the weak convergence of a sequence of the complex Monge-Amp\`ere operator $(\om + dd^c u_j)^n$ to $(\om+ dd^c u)^n$ if we assume that $u_j$ is convergent to $u$ in $Cap_{\om}$. In Section \ref{sec4}, we establish the existence of weak solutions of the Monge-Amp\`ere equations on compact Hermitian manifolds under the assumption that there exists a smooth subsolution. 
\vskip 0.3cm
\n 
\n 
{\bf Acknowledgments.} 
The authors would like to thank the anonymous referees for their valuable comments and suggestions, which led to improvements in the paper's exposition.
The first and second named authors are supported by Grant number B2025-CTT-10 from the Ministry of Education and Training, Vietnam.
This work was written during our visit to the Vietnam Institute for Advanced Study in Mathematics (VIASM) in the Spring of 2025. We also thank VIASM for financial support and hospitality.}

\section{Preliminairies}\label{sec2}
{\fontsize{11 pt}{1}\selectfont	
\n Through this paper by $(X,\om)$ we denote a fixed compact Hermitian manifold $X$ of dimension $n$ equipped with a Hermitian metric $\om$. As mentioned in the beginning of the paper, locally, $\omega$ can be written as
$ \om=\frac{i}{2}\sum\limits_{\alpha,\beta} a_{\alpha,\bar{\beta}} dz_{\alpha}\wed d\bar{z}_{\beta},$
\n where $(a_{\alpha,\bar{\beta}})$ is a positive definite Hermitian matrix. The smooth volume form $dV$ associated to this Hermitian metric is given by the $n$-th wedge product $dV=\om^n$. Moreover, we assume that $ \int\limits_{X}dV<\infty$.
\begin{definition}\label{dn2.1}	
{\fontsize{11 pt}{1}\selectfont		A function $u:X\rightarrow [-\infty,+\infty)$ is call $\om$-plurisubharmonic function (briefly, $\om$-psh) if it satisfies the following conditions:\\	
	\n (i) $u$ is upper semi-continuous on $X$.	\\
	\n (ii) $u\in L^1(X,dV)$ and $\om+dd^c u\geq 0$ in the sense of current. }	
\end{definition}
\n By $PSH(X,\om)$ or $PSH(\om)$ we denote the set of $\om$-psh functions on $X$.
 Now we deal with the complex Monge-Amp\`ere operator on $PSH(\om)$. Let $u\in PSH(\om)\cap L^{\infty}(X)$.  Following the  fundamental work of Bedford and Taylor in \cite{BeTa82}, we know that the complex Monge-Amp\`ere
operator $(\om + dd^c u)^n$ is well-defined. Moreover, it is a positive Borel measure on $X$. For a detailed account of the existence of the operator 
$(\omega + dd^c u)^n$, where $u \in \text{PSH}(\omega) \cap L^{\infty}(X)$, 
we  refer readers to the construction provided before Proposition 2.1 in \cite{DK12} and more detail, in Proposition 1.2 of \cite{KN15}.
\begin{definition}\label{dn2.2}
{\fontsize{11 pt}{1}\selectfont		Let $(u_j)\subset PSH(\om)\cap L^{\infty}(X)$ and $u\in PSH(\om)\cap L^{\infty}(X)$. The  sequence of Monge-Amp\`ere operators $(\om+dd^c u_j)^n$ is called weak convergence to $(\om+ dd^c u)^n$ if for all $\psi\in C^{\infty}_0(X)$ the following condition holds
	$$\lim\limits_{j\to\infty}\int\limits_{X}\psi(\om+dd^c u_j)^n = \int\limits_{X}\psi (\om+ dd^c u)^n.$$  }
\end{definition}

\n where $C^{\infty}_0(X)$ is the set of smooth functions with compact support in $X$. Each $\psi\in C^{\infty}_0(X)$ also is called a test function. When the sequence of Monge-Amp\`ere operators $(\omega + dd^c u_j)^n$ is convergent weakly to $(\omega + dd^c u)^n$, we also say that the Monge-Amp\`ere operator is continuous, where \textquotedblleft continuity \textquotedblright is understood in the sense that it is continuous in the weak*-topology.

\n Now as in \cite{GZ05} we give the notion of capacity $Cap_{\om}$ of a Borel subset $E\subset X$.
\begin{definition}\label{dn2.3}
{\fontsize{11 pt}{1}\selectfont		Assume that $E$ is a Borel subset of $X$. Capacity of $E$ denoted by $Cap_{\om}(E)$ and is defined by
	$$Cap_{\om}(E)=\sup\Bigl\{\int\limits_{E}(\om+dd^c u)^n:  u\in PSH(\om),\  \ 0\leq u\leq 1\Bigl\}.$$  }
\end{definition}	
\n Next, we deal with the convergence in capacity of a sequence $(u_j)\subset PSH(\om)$ to a function $u\in PSH(\om)$ 
\begin{definition}\label{dn2.4}
{\fontsize{11 pt}{1}\selectfont		A sequence $(u_j)$ of $\om$-psh functions on $X$ is said to be convergence in capacity to a $\om$-psh $u$ if for all $\eta>0$ the following condition holds
	$$\lim\limits_{j\to\infty}Cap_{\om}\Bigl\{|u_j - u|> \eta\Bigl\}= 0. $$ }
\end{definition}
\n By repeating the same arguments as in the proof of Corollary 3.8 in \cite{GZ05} and arguments are explained  before  Proposition 2.1 in \cite{DK12}, we note that the quasi-continuity of $\omega$-plurisubharmonic functions  is still valid in Hermitian setting. Namely, we have the following result.
\begin{proposition}\label{md1.1}
{\fontsize{11 pt}{1}\selectfont		Let  $u\in PSH(\om)$ be a $\om$-psh function. Then for each $\varepsilon >0$  there exists an open subset $G\subset X$ with $Cap_{\om}(G) <\varepsilon$ and $u|_{X\setminus G}$ is continuous. }
\end{proposition}
\n Next, we recall the notion of variation of a measure. Given a Borel measure $\mu$ (non-necessarily positive). One defines its variation $\|\mu\|$ on a Borel set $A$ by
$$\|\mu\|(A)=\sup\Bigl\{\sum\limits_{j=1}^{\infty}|\mu(A_j)|\Bigl\},$$

\n where $\{A_j\},j\geq 1$ is any disjoint at most countable partition of the set $A$ by Borel sets. The classical Hahn decomposition theorem states that in fact that it is enough to consider only the special partition given by the two Hahn sets,intuitively,the pieces where $\mu$ is positive and negative.
 To avoid lengthy expressions, we adopt the convention that from now on, when writing the Monge-Amp\`ere operator of a function  $ u \in PSH(\omega) \cap L^{\infty}(X)$, we use the notation $ \omega^n_u$. Thus, we understand $\omega^n_u = (\omega + dd^c u)^n$.   }

\section{Continuity of the complex Monge-Ammp\`ere operator on compact Hermitian manifolds}\label{sec3}
{\fontsize{11 pt}{1}\selectfont	
\n First, we recall the  following class of $\omega$-sh functions introduced and investigated in \cite{KN25}. Namely, let $(X,\om)$ be a compact Hermitian manifold mentioned in Section 2. By $\mathcal{P}_0$ we denote the set of $\om$-psh functions defined by
$$\mathcal{P}_0=\Bigl\{v\in PSH(\om)\cap L^{\infty}(X): \  \ \sup\limits_{X}v=0\Bigl\}.$$
\n  We begin the section by proving a following  result which serves as an essential tool for subsequent results.
\begin{theorem}\label{dl3.1}
{\fontsize{11 pt}{1}\selectfont		Let $(u_j)\subset PSH(\om)\cap \mathcal{P}_0$ be a uniformly bounded sequence of $\om$-psh. Assume that $(u_j)$ is convergent to $u\in \mathcal{P}_0$ in $Cap_{\om}$. Then $\om^n_{u_j}$ is weakly convergent to $\om^n_u$. }
\end{theorem}
\begin{proof}
{\fontsize{11 pt}{1}\selectfont		It is enough to show that every subsequence of the $\om^n_{u_j}$ has its subsequence which is weakly convergent to $\om^n_u$. Therefore, without loss of generality, we may assume that $(u_j)$ is a subsequence of the original sequence. From the hypothesis $(u_j)$ is convergent in $Cap_{\om}$ to $u$, by Proposition 2.5 in \cite{KN25} we have
\begin{equation}\label{3.1}
	\lim\limits_{j\to\infty} \int\limits_{X}|u_j-u| \om^n_{u_j}=0.
	\end{equation}
	
	\n Hence, coupling \eqref{3.1} and Lemma 2.4 in \cite{KN25} it follows that there exists a subsequence $(u_{j_k})$ of the $(u_j)$ such that $\om^n_{u_{j_k}}$ is weakly convergent to $\om^n_u$ and the desired conclusion follows. }
\end{proof}

\n Now we give the following result. Note that this result was mentioned in \cite{GZ05} and \cite{DK12}, but no proof was given there. By relying on the quasi-continuity of $\om$-psh functions and Theorem \ref{dl3.1}, we will provide a proof here in Hermitian setting. Namely, we have 
\begin{corollary}\label{hq3.1}
{\fontsize{11 pt}{1}\selectfont		Let $(u_j)$ be a  uniformly bounded sequence of $\om$-psh functions in $\mathcal{P}_0$. Assume that $(u_j)$ is decreasing to $u\in \mathcal{P}_0$. Then $\om^n_{u_j}$ is weakly convergent to $\om^n_u$. }
\end{corollary}
\begin{proof}
{\fontsize{11 pt}{1}\selectfont		By Theorem \ref{dl3.1} it remains to show that $(u_j)$ is convergent to $u$ in $Cap_{\om}$. For $\varepsilon>0$, by the quasi-continuity of $u$ there exists an open subset $G\subset X$ such that $Cap_{\om}(G)< \varepsilon$ and $u$ is continuous on $X\setminus G$. Note that $X\setminus G$ is compact, by Hartogs lemma for $\eta> 0$ there exits $j_0>0$ such that $\forall\  \ j\geq j_0$ and  for all $ x\in X\setminus G$ we get that
	\begin{equation}\label{3.3}
	u_j(x)\leq u(x)+ \eta.
	\end{equation}
	
	\n On the other hand, by the hypothesis we have 
	$$u_j(x)\geq u(x)> u(x) - \eta,\forall\ \ x\in X.$$
	
	\n This fact together with \eqref{3.3} implies that for all $x\in X\setminus G$ and $j\geq j_0$ we deduce that
	\begin{equation}\label{3.4}
	u(x)-\eta <u_j(x)\leq u(x)+\eta.
	\end{equation}
	\n From \eqref{3.4} it follows that
	$$|u_j(x) - u(x)|<\eta,\forall\ \ j\geq j_0,\  \ x\in X\setminus G.$$
	
	\n Hence, $\Bigl\{|u_j(x)- u(x)|>\eta \Bigl\}\subset G$ for $j\geq j_0.$ This yields $$Cap_{\om}\Bigl(\Bigl\{|u_j(x)- u(x)|>\eta \Bigl\}\Bigl) \leq Cap_{\om}(G)<\varepsilon,$$ for $j\geq j_0$. Such as, the sequence $(u_j)$ is convergent to $u$ in $Cap_{\om}$ as we wanted. The proof is complete.  }
\end{proof}

\n The following result is an extension of Theorem 2 in \cite{Xi08} in  Hermitian setting.
\begin{theorem}\label{dl3.2}
{\fontsize{11 pt}{1}\selectfont		Assume that for all $v\in PSH(\om)\cap L^{\infty}(X)$ the following condition holds
	\begin{equation}\label{3.5}
	\int\limits_{X}\om^n_v=\int\limits_{X}\om^n.
	\end{equation}
	
	\n Let $(u_j)\subset \mathcal{P}_0$ be a uniformly bounded sequence of $\om$-psh functions and $u\in \mathcal{P}_0$. Suppose that $u_j\rightarrow u$ in $L^1(X,dV)$ and for all $\delta>0$, $\lim\limits_{j\to\infty} \int\limits_{\{u\geq u_j+\delta\}}\om^n_{u_j}= 0$. Then $\om^n_{u_j}\rightarrow \om^n_u$ weakly. }
\end{theorem}
\begin{proof}
{\fontsize{11 pt}{1}\selectfont		It is enough to show that any subsequence of the $(\om^n_{u_j})$ has its subsequence which is weakly convergent to $\om^n_u$. So we may assume that $$\int\limits_{\{u\geq u_j+\frac{1}{j}\}} \om^n_{u_j}\rightarrow 0\ \text{as} \ j\to\infty.$$ Put $\phi_j=\max\{u_j,u -\frac{1}{j}\}$. Then $\phi_j\in \mathcal{P}_0$ and the sequence $(\phi_j)$ is uniformly bounded on $X$. Now we prove that $\phi_j$ is convergent to $u$ in $Cap_{\om}$. Indeed, we have\begin{equation}\label{3.6}
	\left|u_j-\left(u-\frac{1}{j}\right)\right| =\phi_j-u_j + \phi_j-\left(u-\frac{1}{j}\right).
	\end{equation}
	
	\n Since $u_j\rightarrow u$ in $L^1(X,dV)$ then $\int\limits_{X}|u_j-u|dV\rightarrow 0$ as $j\to\infty$. On the other hand, we have
	$$\int\limits_{X}\left|u_j-\left(u-\frac{1}{j}\right)\right|dV\leq \int\limits_{X}|u_j-u|dV+ \frac{1}{j} \int\limits_{X}dV\rightarrow 0,\  \ \text{as $j\to\infty$.}$$
	
	\n However, by \eqref{3.6} it follows that
	$$0\leq \phi_j-\left(u-\frac{1}{j}\right)\leq \left|u_j-\left(u-\frac{1}{j}\right)\right|,$$
	
	\n then we get that $\phi_j\rightarrow u$ in $L^1(X,dV)$. By Proposition \ref{md1.1} for $\varepsilon>0$ there exists an open $G\subset X$ with $Cap_{\om}(G)<\varepsilon$ and $u$ is continuous on $X\setminus G$. Using Hartogs lemma for $\eta>0$  we can find $j_0>0$ such that for $j\geq j_0$, $\frac{1}{j}<\eta$ and
	$$ \phi_j\leq u+\eta,$$
	
	\n on $X\setminus G$. But we have $\phi_j\geq u-\frac{1}{j} > u-\eta$ on $X$. It follows that on $X\setminus G$ we get that
	$$|\phi_j - u|<\eta,\  \ \forall j\geq j_0.$$	
	\n Hence, using the arguments in the proof of Corollary \ref{hq3.1} we infer that $\phi_j\rightarrow u$ in $Cap_{\om}$ as $j\to\infty$. Theorem \ref{dl3.1} implies that $\om^n_{\phi_j}\rightarrow \om^n_u$ weakly in $X$ as $j\to\infty$. On the other hand, by the proof of Proposition 3.1 in  \cite{DK12} we get that
	\begin{equation}\label{3.7}
	1_{\{u_j>u-\frac{1}{j}\}}\om^n_{\phi_j}= 1_{\{u_j>u-\frac{1}{j}\}} \om^n_{u_j}.
	\end{equation}
	\n From \eqref{3.7} we infer that
	\begin{align*}
	\om^n_{u_j} - \om^n_{\phi_j}&= 1_{\{u_j\leq u-\frac{1}{j}\}}\Bigl(\om^n_{u_j}- \om^n_{\phi_j}\Bigl)
	= -1_{\{u_j\leq u-\frac{1}{j}\}}\om^n_{\phi_j}+ 1_{\{u_j\leq u-\frac{1}{j}\}} \om^n_{u_j}.
	\end{align*}
	
	\n By the hypothesis $$\int\limits_{\{u_j\leq u-\frac{1}{j}\}}\om^n_{u_j}= \int\limits_{\{u\geq u_j+\frac{1}{j}\}}\om^n_{u_j}\rightarrow 0 \ \text{as}\ j\to\infty.$$ On the other  hand, by using \eqref{3.5} and \eqref{3.7} we get that
	\begin{equation*}\label{3.8}
	\begin{aligned}
	\int\limits_{X}1_{\{u_j\leq u-\frac{1}{j}\}}\om^n_{\phi_j}&=\int\limits_{X}\om^n_{\phi_j}- \int\limits_{\{u_j> u-\frac{1}{j}\}}\om^n_{\phi_j}=\int\limits_{X}\om^n- \int\limits_{\{u_j> u-\frac{1}{j}\}}\om^n_{\phi_j}\\
	&=\int\limits_{X}\om^n_{u_j}- \int\limits_{\{u_j> u-\frac{1}{j}\}}\om^n_{u_j}= \int\limits_{\{u-\frac{1}{j}\geq  u_j\}}\om^n_{u_j}= \int\limits_{\{u\geq u_j +\frac{1}{j}\}}\om^n_{u_j}\rightarrow 0,
	\end{aligned}
	\end{equation*}
	
	\n as $j$ tends to $\infty$ by the hypothesis. Hence, from the arguments presented above, we infer that  $\om^n_{u_j} - \om^n_{\phi_j}\rightarrow 0$ weakly in $X$ as $j\to\infty$. However, we have
	$$\om^n_{u_j} -\om^n_{u}= \om^n_{u_j} - \om^n_{\phi_j} + \om^n_{\phi_j}-\om^n_u\rightarrow 0,$$
	
	\n  weakly in $X$ as $j\to\infty$. Thus we achieve that $\om^n_{u_j}\rightarrow \om^n_u$ weakly as $j\to\infty$ as we wanted. The proof is complete. }
\end{proof}

\n From Theorem \ref{dl3.2} we obtain the following corollary which is an extension of Corollary 2 in \cite{Xi08} in  Hermitian setting.
\begin{corollary}\label{hq3.2}
{\fontsize{11 pt}{1}\selectfont		Let the condition \eqref{3.5} be satisfied and $(u_j)\subset\mathcal{P}_0$ be a uniformly bounded sequence of $\om$-psh functions and $u\in \mathcal{P}_0$, $u_j\rightarrow u$ in $L^1(X,dV)$. If $\om^n_{u_j}= f_j dV$ with $\sup\limits_{j}\int\limits_{X}f_j^pdV<\infty, p>1$. Then $\om^n_{u_j}$ is weakly convergent to $\om^n_u$ as $j\to\infty$ in $X$.  }
\end{corollary}
\begin{proof}
{\fontsize{11 pt}{1}\selectfont		For each $\delta>0$ by using the H\"older inequality we get that 
	\begin{align*}
	\int\limits_{\{u\geq u_j+\delta\}}\om^n_{u_j}= \int\limits_{\{u\geq u_j+\delta\}} f_jdV&\leq \Bigl(\int\limits_{X} f_j^p dV\Bigl)^{\frac{1}{p}}\Bigl(\int\limits_{\{u\geq u_j+\frac{1}{j}\}}dV\Bigl)^{1-\frac{1}{p}}\\
	&\leq \delta^{\frac{1}{p}-1}\sup\limits_{j}\Bigl(\int\limits_{X} f_j^pdV\Bigl)^{\frac{1}{p}}\Bigl(\int\limits_{X}|u_j-u|dV\Bigl)^{1-\frac{1}{p}}\rightarrow 0, 
	\end{align*}
	
	\n as $j\to\infty$ by the hypothesis $u_j\to u$ in $L^1(X,dV)$. Hence, by Theorem \ref{dl3.2} the proof is complete. }
\end{proof}

\n Next, we have the following result. 
\begin{corollary}\label{hq3.3}
{\fontsize{11 pt}{1}\selectfont		Let the condition \eqref{3.5} be satisfied and $(u_j)\subset \mathcal{P}_0$ be a uniformly bounded sequence of $\om$-psh functions and $u\in \mathcal{P}_0$. Assume that $u_j\rightarrow u$ in $L^1(X,dV)$ and there exists $C>0$ such that for all $j\geq 1$, $\om^n_{u_j}\leq CdV$. Then $\om^n_{u_j}\rightarrow \om^n_u$ weakly as $j\to\infty$ in $X$. }
\end{corollary}
\begin{proof}
{\fontsize{11 pt}{1}\selectfont		Indeed, for $\delta>0$ we have
	$$\int\limits_{\{u\geq u_j+\delta\}}\om^n_{u_j}\leq C\int\limits_{\{u\geq u_j+\delta\}} dV\leq \frac{C}{\delta}\int\limits_{X}|u-u_j|dV\rightarrow 0,$$
	
	\n as $j\to\infty$ and  the desired conclusion follows.}
\end{proof}

\n By using the same arguments as in the proof of Corollary \ref{hq3.3} we note that  Corollary \ref{hq3.3} holds for a positive Radon measure $\mu$ on $X$. Namely, we have.
\begin{corollary}\label{hq3.4}
{\fontsize{11 pt}{1}\selectfont		Let the condition \eqref{3.5} be satisfied and $(u_j)\subset \mathcal{P}_0$ be an uniformly bounded sequence of $\om$-psh functions and $u\in \mathcal{P}_0$. Assume $\mu$ is a positive Radon measure on $X$ with $\mu(X)<\infty$. If $u_j\rightarrow u $ in $L^1(X,d\mu)$ and $\om^n_{u_j}\leq Cd\mu, C>0$ for all $j\geq 1$. Then $\om^n_{u_j}$ is weakly convergent to $\om^n_u$ in $X$.  }
\end{corollary}

\n Next, we have a notable result concerning the weak continuity of a sequence of Monge-Amp\`ere operator in Hermitian setting.
\begin{proposition}\label{md3.1}
{\fontsize{11 pt}{1}\selectfont		Let the condition \eqref{3.5} be satisfied and $(u_j)\subset \mathcal{P}_0$ be a uniformly bounded sequence of $\om$-psh functions and $u\in \mathcal{P}_0$. Assume that $u_j\rightarrow u$ in $L^1(X,dV)$ and $\om^n_{u_j}$ vanishes on the set $\{u_j< u\}$ for all $j$. Then $\om^n_{u_j}$ is weakly convergent to $\om^n_u$ in $X$. }
\end{proposition}
\begin{proof}
{\fontsize{11 pt}{1}\selectfont		First, we prove $u_j\geq u$ on $X$ for all $j$. By the hypothesis $u\leq 0$ on $X$. Given $\varepsilon >0$. Choose $1>s>0$ such that $s\max\{1,\sup\limits_{X}|u|\}<\varepsilon$. Assume that $\va \in PSH(\om)$ with $-1\leq\va\leq 0$. Due to the condition \eqref{3.5}, by Proposition 3.1 in \cite{DK12}, the comparison principle holds for bounded $\omega$-plurisubharmonic functions. That is, we have
	\begin{equation}\label{3.9}
	\int\limits_{\{u<v\}}\om^n_v\leq\int\limits_{\{u<v\}}\om^n_u,
	\end{equation}
	
	\n for $u,v\in PSH(\om)\cap L^{\infty}(X)$. Now for all $j\geq 1$, by using the comparison principle \eqref{3.9} we get that
	\begin{align*}
	 s^n\int\limits_{\{u_j< u-2\ve\}}\om^n_{\va}&\leq \int\limits_{\{u_j< u-2\ve\}}(\om+ dd^c ((1-s)u+s\va)))^n\\
	&\leq \int\limits_{\{u_j< (1-s)u+s\va-\ve\}}\Bigl(\om+ dd^c\Bigl[(1- s)u + s\va\Bigl]\Bigl)^n\\
	&\leq \int\limits_{\{u_j< (1-s)u+s\va-\ve\}} (\om + dd^c u_j)^n\leq \int\limits_{\{u_j< u\}} \om^n_{u_j}= 0.
	\end{align*}
	
	\n Taking the supremum over all such $\varphi$, by the definition of $Cap_{\omega}$, we get  that $$Cap_{\om}\Bigl(\{u_j < u-2\ve\}\Bigl) = 0\ \text{for all}\  j.$$ Hence, $u_j\geq u-2\ve$ on $X$ for all $j$. Let $\ve\searrow 0$ we deduce that $u_j\geq u$ on $X$ for all $j$ and the desired conclusion follows. By repeating the same arguments as in the proof of Corollary \ref{hq3.1} we infer that $u_j$ is convergent to $u$ in $Cap_{\om}$. By Theorem \ref{dl3.1}, $\om^n_{u_j}\rightarrow \om^n_u$ weakly on $X$. The proof is complete. }
\end{proof}

\n From the proof of Proposition \ref{md3.1} we also have the following result.
\begin{corollary}\label{hq3.5}
{\fontsize{11 pt}{1}\selectfont		Let the condition \eqref{3.5} be satisfied and $(u_j)\subset\mathcal{P}_0$ be an uniformly bounded sequence of $\om$-psh functions and $u\in \mathcal{P}_0$.  Assume that $u_j\rightarrow u$ in $L^1(X,dV)$ and $Cap_{\om}\Bigl(\{u_j< u\}\Bigl)= 0$ for all $j$. Then $\om^n_{u_j}$ is weakly convergent to $\om^n_u$ in $X$.	}
\end{corollary}

\n The following result is a form of Theorem 3.6 in \cite{DH12} in Hermitian setting.
\begin{proposition}\label{md3.2} 
{\fontsize{11 pt}{1}\selectfont		Let $(u_j)\subset \mathcal{P}_0$ be an uniformly bounded sequence of $\om$-psh functions  and $u\in \mathcal{P}_0$. Assume that the following conditions hold:
	
	\n (i) $u_j\rightarrow u$ in $L^1(X,dV)$
	
	\n (ii) Let $\phi_j=\max\{u_j, u\}$ and assume that
	$$\lim\limits_{j\to\infty}\int\limits_{X}\|\om^n_{u_j}- \om^n_{\phi_j}\| = 0.$$
	
	\n Then $\om^n_{u_j}\rightarrow \om^n_{u}$ weakly.}	
\end{proposition}
\begin{proof}
{\fontsize{11 pt}{1}\selectfont		It is easy to see that $\phi_j\in\mathcal{P}_0$ and $(\phi_j)$ is a  uniformly bounded sequence of $\om$-psh functions. From (i) we deduce \begin{equation}\label{3.10}
	\lim\limits_{j\to\infty}\int\limits_{X}|u_j- u|dV= 0.
	\end{equation}
	
	\n Since $|u_j-u| = (\phi_j- u_j) + (\phi_j - u)$ then it follows that $0\leq \phi_j- u\leq |u_j - u|$. Using \eqref{3.10}, we infer that 
	$$0\leq \int\limits_{X}(\phi_j - u)dV\leq\int\limits_{X}|u_j - u|dV\rightarrow 0,$$
	
	\n as $j\to\infty$. This yields that $\phi_j\rightarrow u$ in $L^1(X,dV)$. However, $\phi_j\geq u$ then by repeating the same arguments as in the proof of Corollary \ref{hq3.1} we get that $\phi_j$ is convergent to $u$ in $Cap_{\om}$. Theorem \ref{dl3.1} implies that $\om^n_{\phi_j}\rightarrow \om^n_{u}$ weakly in $X$. We obtain
	$$\om^n_{u_j}- \om^n_{u}= \om^n_{u_j} - \om^n_{\phi_j} + \om^n_{\phi_j} -  \om^n_{u}.$$
	
	\n Hence, it remains to prove that $ \om^n_{u_j}- \om^n_{\phi_j}\rightarrow 0$ weakly in $X$. Let $\psi\in C^{\infty}_0(X)$ be an arbitrary test function. Then we have
	\begin{equation}\label{3.12}
	\Bigl|\int\limits_{X}\psi(\om^n_{u_j}- \om^n_{\phi_j})\Bigl|\leq \int\limits_{X}|\psi|\|\om^n_{u_j}- \om^n_{\phi_j}\|\leq\sup\limits_{X}|\psi|\int\limits_{X}\|\om^n_{u_j}-\om^n_{\phi_j}\|.
	\end{equation}
	\n Because of the hypothesis (ii) the right-hand side of \eqref{3.12} tends to $0$ as $j\to\infty$ and the desired conclusion follows. }
\end{proof}

\n Next, we give a following result which is an extension of Corollary 1 in \cite{Xi08} in Hermitian setting. Namely, we have the following result.

\begin{proposition}\label{md3.3}
{\fontsize{11 pt}{1}\selectfont		Let $(u_j)\subset\mathcal{P}_0$ be a uniformly bounded sequence and $u\in \mathcal{P}_0$. Assume that $u_j\rightarrow u$ in $L^1(X,dV)$ and there exists a positive Radon measure $\mu$ on $X$ vanishing on pluripolar sets. If $\om^n_{u_j}\leq\mu$ for all $j\geq 1$ then $\om^n_{u_j}\rightarrow \om^n_u$ weakly in $X$. }
\end{proposition}	
\begin{proof}
{\fontsize{11 pt}{1}\selectfont		It is enough to show that every subsequence of the $\om^n_{u_j}$ has its subsequence which is weakly convergent to $\om^n_u$. Therefore, without loss of generality, we consider $(u_j)$ as a subsequence of the original sequence. Because $u_j\rightarrow u$ in $L^1(X,dV)$ then there exists a subsequence $(u_{j_k})\subset (u_j)$ such that $u_{j_k}$ converges $dV$-a.e. to $u$. Then by Corollary 2.2 in \cite{KN25} we can choose a subsequence of the $(u_{j_k})$ which we still denote it by $(u_{j_k})$ satisfying 
	$$\lim\limits_{k\to\infty}\int\limits_{X}|u_{j_k} - u|d\mu=0.$$
	
	\n For an arbitrary $\delta>0$ consider
	\begin{align*}
	\int\limits_{\{u\geq u_{j_k}+\delta\}}\om^n_{u_{j_k}}&\leq\frac{1}{\delta}\int\limits_{\{u\geq u_{j_k}+\delta\}}|u_{j_k}-u|\om^n_{u_{j_k}}\\
	&\leq \frac{1}{\delta}\int\limits_{X}|u_{j_k}-u|\om^n_{u_{j_k}}\leq \frac{1}{\delta} \int\limits_{X}|u_{j_k} - u|d\mu\rightarrow 0,
	\end{align*}

	\n as $k\to \infty$. By using Theorem \ref{dl3.2} we infer that $\om^n_{u_{j_k}}\rightarrow\om^n_u$ weakly in $X$. The  proof is complete.}
\end{proof}

}
\section{Applications}\label{sec4}
{\fontsize{11 pt}{1}\selectfont	
\n In this section by relying on the main result in \cite{KN25} we give a result about solving Monge-Amp\`ere equation on compact Hermitian manifolds if there exists a smooth subsolution. Note that, for the case of a compact Hermitian manifold with boundary and with given boundary conditions, if we have the assumption about the existence of a subsolution, then the existence of a solution to the Dirichlet problem for the complex Monge-Amp\`ere equation has been recently established in \cite{KN23}. Now we state and prove the following result.
\begin{theorem}\label{dl4.1}{\fontsize{11 pt}{1}\selectfont	 Let $(X, \omega)$ be a compact Hermitian manifold of complex dimension $n$, with smooth volume form $dV = \omega^n$ as in Section \ref{sec3}. Suppose that $\mu$ is a positive Borel measure with $\mu(X) = \int\limits_X \om^n = \int\limits_X dV$ and there exists a function $v \in PSH(\omega) \cap C^{\infty}(X)$ such that $\mu \leq A \omega_v^n$ for some constant $A > 0$. Then, there exists a function $u \in PSH(\omega) \cap L^{\infty}(X)$ and a constant $c> 0$ such that $\omega^n_u = c\mu$.}
\end{theorem}
\begin{proof}
{\fontsize{11 pt}{1}\selectfont	
	Using Lebesgue-Radon-Nikodym theorem we can write $\mu= f\om^n_v$, $f\in L^1(X,\om^n_v), 0\leq f\leq A$. Next, we choose a sequence $(f_j)\in L^1(X,\om^n_v)\cap C^{\infty}(X)$ such that $0<f_j\leq 2A$ for all $j\geq 1$ and $\int\limits_{X}|f_j - f|\om^n_{v}\rightarrow 0$ as $j\to\infty$.
	Since the measure $f_j\omega^n_{v}$ is smooth, by a well-known result of \cite{TW10}, there exists a sequence $(u_{j})\in PSH(\om)\cap C^{\infty}(X), \sup\limits_{X} u_{j}=0$ and $B_{j}>0$ such that 
	\begin{equation}\label{4.1}
	\om^n_{u_{j}}= B_{j} f_j\om^n_{v}.
	\end{equation}
	According to Lemma 5.9 in \cite{KN15}, we infer that $0<B_{j}\leq B$ with a uniform $B>0.$
	Therefore, it follows from Corollary 5.6 in \cite{KN15} that we have $-C\leq u_j\leq 0$ with uniform $C>0.$ By passing to a subsequence we may assume that $u_j\to u$ in $L^1(X, dV)$, $u\in PSH(\om)$, $-C\leq u\leq 0$ and $B_j\to c\geq 0.$  We have\begin{equation}\label{4.2}
	\om^n_{u_{j}}= B_{j} f_j\om^n_{v}\leq 2AB\om^n_{v}.
	\end{equation}
	From \eqref{4.2} and utilizing the main result in \cite{KN25}, there is a subsequence of $\{u_j\},$ for simplicity still denoted by $\{u_j\}$ such that $\omega^n_{u_j}\to \omega^n_u$ as  $j\to\infty.$ By Equation \eqref{4.1} we infer that $$\omega^n_u=cf\omega^n_v=c\mu.$$ Since $u$ is bounded, by Remark 5.7 in \cite{KN15} it follows that $c>0$.
	The proof of theorem is complete. }
\end{proof}
}
\vskip 0.3 cm

\begin{thebibliography}{9999}
	{\fontsize{10 pt}{1}\selectfont	
\bibitem{BeTa82} E. Bedford and B. A. Taylor. A new capacity for plurisubharmonic functions. Acta Math., \textbf{149}:1--40, 1982. \url{https://doi.org/10.1007/BF02392348}
	
\bibitem{BL99} T. Bloom and N. Levenberg. Capacity convergence results and applications to a Berstein-
Markov inequality. Trans. Amer. Math. Soc., \textbf{351}(12):4753--4767, 1999. \url{https://doi.org/10.1090/S0002-9947-99-02556-8}

\bibitem{Ce83} U. Cegrell. Discontinuit{\'e} de l'op{\'e}rateur de {M}onge-{A}mp{\`e}re complexe. C.R Acad. Sci. Paris S{\'e}r. I Math., \textbf{296}(21):869--871, 1983.

\bibitem{Ce98} U. Cegrell. Pluricomplex energy. Acta Math., \textbf{180}(2):187--217, 1998. \url{https://doi.org/10.1007/BF02392899}


\bibitem{Ce01} U Cegrell. Convergence in capacity. Isaac Newton Institute for Math. Science Preprint Series NI01046-NPD, 2001, also available at arxiv.org: math. CV/0505218, 2001.

\bibitem{Ce04} U. Cegrell. The general definition of the complex {M}onge-{A}mp{\`e}re operator. Ann. Inst. Fourier
(Grenoble), \textbf{54}(1):159--179, 2004. \url{https://doi.org/10.5802/aif.2014}


\bibitem{Ce12} U. Cegrell. Convergence in capacity. Can. Math. Bull., \textbf{55}(2):242--248, 2012. \url{https://doi.org/10.4153/CMB-2011-078-6}


\bibitem{DH12}  S. Dinew and Pham Hoang Hiep. Convergence in capacity on compact {K}{\"a}hler manifolds. Ann. Sc. Norm. Super. Pisa-Cl. Sci., \textbf{11}(4):903--919, 2012. \url{https://www.numdam.org/item/ASNSP_2012_5_11_4_903_0/}


\bibitem{DK12} S. Dinew and S. Ko{\l}odziej. Pluripotential estimates on compact Hermitian manifolds. Series: Advanced
Lectures in Mathematics, Book title: Advances in geometric analysis : Workshop in honour of
Shing-Tung Yau's 60th birthday, Warsaw, Poland, April 6-8,(2009). Number 21, p.69--86. Somerville: International Press, 2012. \url{http://ruj.uj.edu.pl/xmlui/handle/item/599}


\bibitem{GZ05} V. Guedj and A. Zeriahi. Intrinsic capacities on compact {K}{\"a}hler manifolds. J. Geom. Anal., \textbf{15}:607--639, 2005. \url{https://doi.org/10.1007/BF02922247}

\bibitem{GZ07} V. Guedj and A. Zeriahi. The weighted {M}onge-{A}mp{\`e}re energy of quasiplurisubharmonic
functions. J. Funct. Anal., \textbf{250}(2):442--482, 2007. \url{https://doi.org/10.1016/j.jfa.2007.04.018}


\bibitem{Hi08a} Pham Hoang Hiep. On the convergence in capacity on compact {K}{\"a}hler manifolds and its
applications. Proc. Amer. Math. Soc., \textbf{136}(6):2007--2018, 2008. \url{https://doi.org/10.1090/S0002-9939-08-09043-6}


\bibitem{Hi08} Pham Hoang Hiep. Convergence in capacity. Ann. Pol. Math., \textbf{93}:91--99, 2008. \url{http://pldml.icm.edu.pl/pldml/element/bwmeta1.element.bwnjournal-article-doi-10_4064-ap93-1-8}

\bibitem{Hi10} Pham Hoang Hiep. Convergence in capacity and applications. Math. Scand., \textbf{107}(1):90--102, 2010. \url{https://www.jstor.org/stable/24493697}

\bibitem{KN15} S. Ko{\l}odziej and N.-C. Nguyen,  Weak solutions to the complex {M}onge-{A}mp{\`e}re equation on Hermitian manifolds. Analysis, complex geometry, and mathematical physics: in honor of Duong H. Phong, Contemporary Mathematics, vol. {\bf 644} (American Mathematical Society, Providence, RI, 2015), 141--158.  \url{http://dx.doi.org/10.1090/conm/644/12775}


\bibitem{KN23} S. Ko{\l}odziej and N.-C. Nguyen, The {D}irichlet problem for the {M}onge-{A}mp\`ere equation on {H}ermitian manifolds with boundary,  Calc. Var. Partial Differ. Equ., \textbf{62}(1), 2023. \url{https://doi.org/10.1007/s00526-022-02336-y}


\bibitem{KN25} S. Ko{\l}odziej and N.-C. Nguyen,  Weak convergence of {M}onge-{A}mp{\`e}re measures on compact Hermitian
manifolds. In: Sugawa, T., et al. Complex Geometric Analysis. CCGA. Springer Proceedings in
Mathematics Statistics, vol 481, pages 113--123. Springer, Singapore., 2022. \url{https://doi.org/10.1007/978-981-96-0447-0_10}



\bibitem{TW10} V. Tosatti and B. Weinkove. The complex {M}onge-{A}mp{\`e}re equation on compact {H}ermitian manifolds. J. Amer. Math. Soc., \textbf{23}(4):1187--1195, 2010. \url{https://doi.org/10.1090/S0894-0347-2010-00673-X}


\bibitem{Xi96} Y. Xing. Continuity of the complex {M}onge-{A}mp\`ere operator. Proc. Amer. Math. Soc., \textbf{124}(2):457--467,
1996. \url{https://doi.org/10.1090/S0002-9939-96-03316-3}


\bibitem{Xi08} Y. Xing. Convergence in capacity. Ann. Inst. Fourier (Grenoble), \textbf{58}(5):1839--1861, 2008. \url{https://doi.org/10.5802/aif.2400}


\bibitem{Xi09} Y. Xing. Continuity of the complex {M}onge-{A}mp{\`e}re operator on compact {K}{\"a}hler manifolds. Math.
Z., \textbf{263}(2):331--344, 2009. \url{https://doi.org/10.1007/s00209-008-0420-8} 







}
\end{thebibliography}
\noindent {\bf Data Availibility Statement}\\
The manuscript has no associated data.
\vskip 0.3 cm

\noindent {\bf Ethics declarations}\\
\noindent {\bf Conflict of interest}\\
The authors declare that no conflict of interest in this paper is reported.
\vskip 0.3 cm

\noindent \textbf{ORCID}\\
\noindent \textsc{Le Mau Hai}  \url{https://orcid.org/0000-0002-2264-8305}\\
\noindent \textsc{Nguyen Van Phu}
\url{https://orcid.org/0000-0002-2851-7250} \\
\noindent \textsc{Trinh Tung}  \url{https://orcid.org/0000-0002-0681-7447}

\end{document}